\documentclass[12pt,twoside,reqno]{amsart}

\usepackage[colorlinks=true,citecolor=blue]{hyperref}
\usepackage{mathptmx, amsmath, amssymb, amsfonts, amsthm, mathptmx, enumerate, color,mathrsfs}
\usepackage{graphicx}

\usepackage{multirow}
\usepackage{multicol}
\usepackage{algorithm}
\usepackage{algorithmic}
\usepackage{epstopdf}
\usepackage{tikz}

\newtheorem{theorem}{Theorem}[section]
\newtheorem{lemma}[theorem]{Lemma}
\newtheorem{proposition}[theorem]{Proposition}
\newtheorem{corollary}[theorem]{Corollary}

\theoremstyle{definition}
\newtheorem{definition}[theorem]{Definition}

\newtheorem{remark}[theorem]{Remark}
\numberwithin{equation}{section}

\begin{document}
\setcounter{page}{1}

\vspace*{2.0cm}
\title[Goldberg-Thorp example in light of ill-posedness classification]
{A note on the Goldberg-Thorp example in light of the classification of linear ill-posed problems in Banach spaces}
\author[Bernd Hofmann and Jens Flemming]{Bernd Hofmann$^1$, Jens Flemming$^2$}
\date{February 28, 2026}
\maketitle
\vspace*{-0.6cm}

\begin{center}
{\footnotesize
$^1$Faculty of Mathematics, Chemnitz University of Technology, 09107 Chemnitz, Germany.\\
$^2$Faculty of Informatics/Mathematics, HTWD -- University of Applied Sciences, 01069 Dresden, Germany.
}\end{center}

\medskip

\vskip 4mm {\footnotesize \noindent {\bf Abstract.}
This note considers the strictly singular mapping, denoted by $B$, from $\ell^1$ onto $\ell^2$ of an example by Goldberg and Thorp from 1963 as a typical hybrid-type operator in the context of the classification of ill-posed linear operators in infinite-dimensional Banach spaces. The null-spaces of hybrid-type operators are not complemented and therefore need special attention. More generally, a given well-posedness definition for linear operators requiring both closed range and complemented null-space is motivated by the continuity of occurring pseudo-inverse operators as a stability criterion. With respect to the operator $B$, structure, representation and properties of the operator and its adjoint are summarized in a theorem. Moreover, limitations and opportunities of regularization approaches for the treatment of $B$ are outlined.

\bigskip


\medskip


}

\renewcommand{\thefootnote}{}
\footnotetext{
E-mail addresses: hofmannb@mathematik.tu-chemnitz.de (B. Hofmann) and jens.flemming@htw-dresden.de (J. Flemming).
}

\bigskip
	
\section{Introduction} \label{sec:introduction}

In the frequently cited work \cite{GoldThorp63} by S. Goldberg and E. O. Thorp, an interesting \emph{strictly singular} and continuous linear operator between sequence spaces is mentioned by their example.
This bounded linear operator, which we will henceforth call $B$,  maps the non-reflexive Banach space $\ell^1$ \emph{onto} the separable Hilbert space $\ell^2$.
The existence of such a \emph{surjective} operator $B: \ell^1 \to \ell^2$ is justified in \cite{Mazur33}. As we can see below, $B$ is not uniquely determined, but always \emph{non-injective},  and in the sequel we will use as $B$ a well-selected representative from the family of bounded linear operators from $\ell^1$ onto $\ell^2$.

The more in-depth characterization of the operator $B$ fits with $A:=B,\;X:=\ell^1$ and $Y:=\ell^2$ into the treatment of linear operator equations
\begin{equation}\label{eq:opeq}
A\, x =y  \quad (x \in X,\; y \in Y)\,,
\end{equation}
where $X$ and $Y$ are infinite-dimensional Banach spaces and $A: X \to Y$ denotes a bounded linear operator mapping between $X$ and $Y$. In particular, the distinction between \emph{well-posed}
and \emph{ill-posed} equations \eqref{eq:opeq}, respectively of the corresponding operators $A$, plays a prominent role for the characterization. M.~Z.~Nashed has introduced in the seminal paper \cite{Nashed86}
an approach to classification and distinguishing type~I and type~II of ill-posedness for such equations. \linebreak As in \cite{Nashed86}, in much of the literature on linear problems in Banach spaces with continuous operators, well-posedness is defined by the occurrence of a closed range $\mathcal{R}(A)$ of the operator $A$, whereas a non-closed range $\mathcal{R}(A)$ indicates ill-posedness. This would mean well-posedness for the specific operator equation
\begin{equation}\label{eq:Beq}
B\, x =y  \quad (x \in \ell^1,\; y \in \ell^2)
\end{equation}
under discussion here with the closed range $\mathcal{R}(B)=\overline{\mathcal{R}(B)}^{\ell^2}=\ell^2$. Indeed, the alleged well-posedness only based on the property of a closed range is motivated by the fact that a closed range implies \emph{stability} in the sense of V.~K.~Ivanov \cite{Ivanov63} also for non-injective bounded linear operators $A$. This means, for any exact right-hand side $y \in \mathcal{R}(A)=\overline{\mathcal{R}(A)}^Y$, that approximations $y_n \in \mathcal{R}(A)$ with
$\lim \limits_{n \to \infty} \|y_n-y\|_Y=0$ imply the convergence of the non-symmetric \emph{quasi-distance} ${\rm qdist}(A^{-1}(y_n),A^{-1}(y)):= \sup \limits_{\tilde x \in A^{-1}(y_n)} \inf \limits _{x \in A^{-1}(y)} \|\tilde x -x\|_X \to 0$ as $n \to \infty$, and we refer for details to
\cite[Prop.~1.12]{Flemmingbuch18}. However, this kind of stability with the rather weak quasi-distance error measure is only one possible aspect for a well-posedness definition.

Another aspect of stability is the selection of ``well-behaving'' concrete approximate solutions $x^\delta=x^\delta(y^\delta) \in X$ to equation \eqref{eq:opeq} with fixed right-hand side $y \in \mathcal{R}(A)$ for any given noisy data $y^\delta \in Y$, where $\|y^\delta-y\|_Y \le \delta$ and $\delta>0$ expresses the noise level. For the sake of simplicity, we restrict here to the case of bounded linear operator $A: X \to Y$, for which the \emph{range} $\mathcal{R}(A)$  \emph{is dense in} $Y$, i.e.~$\overline{\mathcal{R}(A)}^Y=Y$. There we have first the case that $A$ is \emph{surjective} with $\mathcal{R}(A)=Y$, hence all noisy data $y^\delta \in Y$ are range elements. This case applies for the operator $B$ from the Goldberg-Thorp example.
Secondly, we have the case that $\mathcal{R}(A)$ is a proper subset of $Y$ and noisy data $y^\delta \in Y$ need not belong to the range $\mathcal{R}(A)$.
Brief remarks on the remaining case that the closure of the range  $\overline{\mathcal{R}(A)}^Y$ is a \emph{proper subset} of $Y$ can be found in the appendix.

First we consider the case of \emph{surjective} operators $A$. A standard approach for finding approximate solutions $x^\delta$ is the discrepancy norm minimization
\begin{equation} \label{eq:dismin}
 \|Ax-y^\delta\|_Y \to \min, \quad \mbox{subject to} \; x \in X,
\end{equation}
which is in particular well-regarded by finding least-squares solutions in Hilbert spaces $Y$. There the convex discrepancy norm functional is solvable for all $y^\delta \in Y$. But the minimizers $x^\delta \in X$ to \eqref{eq:dismin} are not uniquely determined if $A$ is non-injective.
Then solutions with specific properties  are being sought. Most popular is the approximation based on the \emph{minimum-norm solution} $x_{mn}$ satisfying the property  $\|x_{mn}\|_X = \min \{\|x\|_X:\, x \in X,\, Ax=y\}$. Such minimum-norm solutions exist and are uniquely determined if $X$ is a smooth and uniformly convex Banach space (cf.~\cite[Lemma~3.3]{Schusterbuch12}). Then good approximations to $x_{mn}$ can be calculated by means of \emph{regularized solutions} $x_\alpha^\delta$ with regularization parameters $\alpha>0$. For the general theory of regularized solutions we refer, for example, to the extensive work of A.~N.~Tikhonov and A.~G.~Yagola, see e.g.~\cite{TikLeoYag98} and \cite{YaKo13}. Some prominent variant of regularization consists in finding $x_\alpha^\delta$ by solving the convex optimization problem
\begin{equation} \label{eq:regmin}
 \|Ax-y^\delta\|^p_Y+\alpha \|x\|^q_X \to \min, \quad \mbox{subject to} \; x \in X,
\end{equation}
for exponents $p,q \ge 1$ and under appropriate choice rules of the regularization parameter $\alpha$. Note that the assignment $y^\delta \in Y \mapsto x_{mn} \in X$ is in general a \emph{nonlinear} procedure.
Unfortunately, the Banach space $X=\ell^1$ under focus here is not uniformly convex and minimum-norm solutions $x_{mn}$ do also for surjective $A:\ell^1 \to Y$ not necessarily exist.
However, as a consequence of the Bartle-Graves Theorem (see \cite[Cor.~5J.4, p.~345]{DoRo14}) there is for surjective $A: \ell^1 \to Y$ a \emph{continuous} operator $S:Y \to \ell^1$
in the sense of a (in general) \emph{nonlinear} recovery operator obeying $ASy=y$ for all $y \in Y$, which could deliver stable approximate solutions, but $S$ is not available in explicit form and hence cannot be exploited for practical use.

On the other hand, for any bounded linear operator $A: X \to Y$ there is a subspace $U$ of $X$ as a complement of the null-space $\mathcal{N}(A)$ of $A$ such that
the direct sum  $X=\mathcal{N}(A) \oplus U$ takes place. Since the restriction of $A|_U: U \to Y$ of the operator $A$ is injective, there is a well-defined \emph{pseudo-inverse} $A_U^\dagger: \mathcal{R}(A) \to U$, and we have to distinguish the cases of complemented and uncomplemented null-spaces.
If there is some closed subspace $U$ of $X$ as a complement to $\mathcal{N}(A)$ in $X$, the null-space is called (topologically) \emph{complemented} in $X$. Then, for
surjective operators $A$ that always possess a closed range, the pseudo-inverse $A_U^\dagger: Y \to U$ is a \emph{bounded} linear operator as a consequence of the open mapping theorem. Otherwise, the null-space is called (topologically) \emph{uncomplemented} in $X$, which means that all subspaces $U$ of $X$ in the direct sum are not closed. This is only possible in infinite dimensional spaces $X$ and if the dimension and the codimension of the null-space $\mathcal{N}(A)$ are not finite. In the monograph \cite{Flemmingbuch18} of the second author it was proven with Proposition~1.11 ibid that then the pseudo-inverse $A_U^\dagger: Y \to U$ is always an \emph{unbounded} linear operator. But in both cases and for all $y \in Y=\mathcal{R}(A)$ there is a well-defined element
$x_u=A_U^\dagger y \in U$, which we call \emph{$U$-solution} of \eqref{eq:opeq} to the exact right-hand side $y$. Then also approximations $x^\delta_u=A_U^\dagger y^\delta \in U$ are well-defined for surjective $A$. There we have for complemented null-space $\mathcal{N}(A)$ an error estimate $\|A_U^\dagger y^\delta-A_U^\dagger y\|_X \le C \,\delta$ (with the operator norm $C=\|A_U^\dagger\|_{\mathcal{L}(Y,X)}$ as constant) that characterizes the quality of the approximation of the $U$-solution to the exact right-hand side $y$ by using the pseudo-inverse operator applied to noisy data $y^\delta \in Y$ with noise level $\delta>0$. The error tends to zero as $\delta \to 0$. The situation is completely different for an uncomplemented null-space $\mathcal{N}(A)$. Then due to the unboundedness of the pseudo-inverse operator there is no such constant $C$ and the two elements $A_U^\dagger y^\delta$ and $A_U^\dagger y$ can be any distance apart even if $\delta>0$ is arbitrarily small. This gives enough motivation to require the property of complementedness of the null-space $\mathcal{N}(A)$ of $A$ in $X$ in addition to the closedness of the range $\mathcal{R}(A)$ for the definition of well-posedness of an operator equation \eqref{eq:opeq}. Such a conclusion was already applied in Definition~4.1 of \cite{HofKin25} and slightly amended in Definition~2.6 of \cite{FleHof25}. In this sense, we recall in Section~\ref{sec:def} the latter definition, including also a distinction of type~I and type~II ill-posedness, which differs slightly from the corresponding definitions suggested in \cite{Nashed86} and more recently in \cite{HofKin25} for good reason.

Now we still briefly consider the case of \emph{non-surjective} operators $A$ with range $\mathcal{R}(A)$ dense in $Y$. This makes the operator equation \eqref{eq:opeq} \emph{ill-posed} in the sense of Nashed \cite{Nashed86}. As a special feature of the non-closed range we have here that, for noisy data  $y^\delta \notin \mathcal{R}(A)$, the discrepancy norm minimization systematically fails. This means that the optimization problem \eqref{eq:dismin} has no minimizers, because the norms $\|x_n\|_X$ of minimizing sequences $(x_n)_{n \in \mathbb{N}}$ characterized by $\lim_{n \to \infty}\|Ax_n-y^\delta\|_Y = \inf_{x \in X}\|Ax-y^\delta\|_Y$ tend to infinity as $n \to \infty$. We note that $U$-solutions $x_u=A_u^\dagger y$ for $y \in \mathcal{R}(A)$ cannot be approximated
by $x^\delta_u=A_u^\dagger y^\delta$ because the pseudo-inverse operator $A^\dagger_U: \mathcal{R}(A) \to U$ is not defined for elements $y^\delta \notin \mathcal{R}(A)$, independent of the question whether the null-space $\mathcal{N}(A)$ is complemented or uncomplemented in $X$.

The remaining paper is organized as follows: Section~\ref{sec:def} presents and motivates a well-posedness definition and the distinction between two types of ill-posedness for linear operator equations in infinite-dimensional Banach spaces. Hybrid-type operators are considered there as a preparation to Section~\ref{sec:Goldberg}, where Mazur-type operators are introduced with focus on the operator of the Goldberg-Thorp example, which maps $\ell^1$ onto $\ell^2$ and will be denoted by $B$. The properties of $B$ are summarized in a theorem in Section~\ref{sec:Goldberg}. Limitations and opportunities of regularization approaches for the treatment of $B$ are outlined in Section~\ref{sec:regu}. An appendix with additional remarks completes the paper.

\section{Characterization of well-posedness and ill-posedness with focus on the hybrid case} \label{sec:def}

We deliver now the present version of a definition for well-posedness
 and ill-posedness characterization and classification of operator equations \eqref{eq:opeq} with bounded linear operators $A: X \to Y$ mapping between infinite dimensional Banach spaces $X$ and $Y$. For simplicity, we assume that the closure of the range of $A$ and the Banach space $Y$ coincide. The associated Figure~\ref{fig:update} is intended to illustrate the definition.

\begin{definition}[Well- and ill-posedness characterization and classification] \label{def:new}
Let $A: X \to Y$ be a bounded linear operator mapping between the
infinite-dimensional Banach spaces $X$ and $Y$.

Then the operator equation \eqref{eq:opeq} is called \emph{well-posed} if
\begin{align*} &\text{the range $\mathcal{R}(A)$ of $A$ is a \emph{closed} subset of $Y$ and, moreover,} \\
&\text{the null-space $\mathcal{N}(A)$ is \emph{complemented} in $X$;}
\end{align*}
otherwise the equation (\ref{eq:opeq}) is called \emph{ill-posed}.

\medskip

In the ill-posed case, (\ref{eq:opeq}) is called \emph{ill-posed of type I}   if
\begin{align*}
&\text{the range $\mathcal{R}(A)$ contains an \emph{infinite-dimensional closed subspace} of $Y$;}
\end{align*}
otherwise the ill-posed
equation (\ref{eq:opeq}) is called \emph{ill-posed of type~II}.
\end{definition}

\begin{figure}[ht]
\begin{center}
\includegraphics{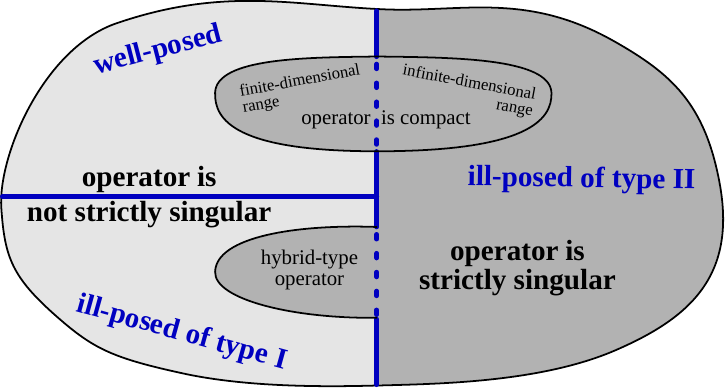}
\caption{Case distinction for bounded linear operators between
infinite-dimensional Banach spaces with dense range} \label{fig:update}
\end{center}
\end{figure}

\pagebreak

As we have motivated in the introduction, well-posedness in the sense of Definition~\ref{def:new} needs a closed range $\mathcal{R}(A)$ as well as a complemented null-space $\mathcal{N}(A)$ of the operator $A$, and we refer for further justification also to Proposition~\ref{pro:appendix} in the appendix.
Although the phenomenon of uncomplementedness was already mentioned in Nashed's publication \cite{Nashed86}, its impact on stability of approximate solutions to the operator equation \eqref{eq:opeq} was not mentioned ibid. Therefore, our new definition differs in this point for good reason from the setting in \cite{Nashed86}. Well-posed
situations can be found in Figure~\ref{fig:update} in the area above the horizontal line left of the vertical line. They include the case of compact operators $A$ with finite-dimensional
range.

In another point, our definition differs also from \cite{Nashed86} as well as from \cite[Def.~4.1]{HofKin25}. We collect in the class of type~I ill-posed problems all ill-posed equations, where the range of $A$ contains an infinite-dimensional closed subspace to be found in Figure~\ref{fig:update} in the area below the horizontal line left of the vertical line. This includes the \emph{hybrid case} along the lines of Definition~4.6 from \cite{HofKin25}, which we will recall in Definition~\ref{def:hybrid}, after we recalled in Definition~\ref{def:ss} the concept of strict singularity.

\begin{definition}[Strictly singular operators] \label{def:ss}
A bounded linear operator $A: X \to Y$ mapping between the Banach spaces $X$ and $Y$  is said to be \emph{strictly singular} if,
given any closed infinite-dimensional subspace $Z$ of $X$, $A$ restricted to $Z$ is
never an isomorphism.
\end{definition}

\begin{definition}[Hybrid-type operators] \label{def:hybrid}
We characterize the operator equation \eqref{eq:opeq} and its corresponding bounded linear operator $A: X \to Y$  as of \emph{hybrid-type} if $A$ is strictly singular and its range $\mathcal{R}(A)$ contains an infinite-dimensional closed subspace of $Y$.
\end{definition}

The following proposition along the lines of Prop.~4.7 from \cite{HofKin25}, but with a slightly amended proof, establishes a connection to the uncomplementedness of the null-space in the hybrid situation and explains the clear area separation between hybrid-type and compact operators in  Figure~\ref{fig:update}.

\begin{proposition} \label{pro:hybridprop}
For an operator equation \eqref{eq:opeq} of hybrid-type, the operator $A: X \to Y$ is not compact, and its null-space $\mathcal{N}(A)$ is always uncomplemented.
\end{proposition}
\begin{proof}
By Definition~\ref{def:hybrid}, $\mathcal{R}(A)$ contains a closed infinite-dimensional subspace, say $M$. Since $M$ is infinite-dimensional, then by the Riesz lemma there is a bounded sequence $(y_n) \subset M$ which has no convergent subsequence. For the non-injective operator $A$, the
open mapping theorem in the formulation of \cite[Theorem~3, item 4]{Tao09b} gives a sequence $(x_n)_{n\in\mathbb{N}}$ with $A x_n = y_n$ and $\|x_n\|_X \leq C \|y_n\|_Y$ for all $n\in\mathbb{N}$. If $A$ was compact, then $(y_n)_{n\in\mathbb{N}}$ would have a convergent subsequence, which is a contradiction.

If the null-space would be complemented for the non-compact operator $A$, then $X=\mathcal{N}(A) \oplus U$ is connected with an infinite-dimensional closed subspace $U$ of $X$.
Now, let again $M \subseteq \mathcal{R}(A)$ denote an infinite-dimensional closed subspace of $Y$ and, due to the continuity and injectivity of $A$ on $U$, the image set $A_U^{\dagger}[M]$
of the continuous pseudo-inverse with respect to $U$ applied to $M$ is an infinite-dimensional closed subspace of $X$, on which the continuous operator $A$ is even continuously invertible.
This, however, contradicts the strict singularity of $A$.
\end{proof}

From the properties of uncomplemented null-spaces in infinite-dimensional Banach spaces we immediately find the assertion of the following corollary.

\begin{corollary} \label{cor:hybrid}
For an operator equation \eqref{eq:opeq} of hybrid-type, the operator $A$ is non-injective and both the dimension and the codimension of its null-space $\mathcal{N}(A)$ are not finite.
\end{corollary}

The crucial point of hybrid-type operators $A$ is its strict singularity (cf.~\cite[p.~284]{Kato58}), which means that for any closed infinite-dimensional subspace $M$ of $X$,
the restriction $A|_M$ of $A$ to $M$ is not an isomorphism. Note that in all ill-posed cases, apart from hybrid-type operators, strictly singular operators $A$ (areas with dark background in Figure~\ref{fig:update}) are ill-posed of type~II (occurring in the area right of the vertical line in Figure~\ref{fig:update}). They include the class
of compact operators with infinite-dimensional range.

\section{The operator of the Goldberg-Thorp example} \label{sec:Goldberg}

For the operator $B$ mapping from $\ell^1$ onto $\ell^2$ of the Goldberg-Thorp example in \cite{GoldThorp63} there is no specific structure prescribed. However, such operator $B$ can
be represented based on operators defined below in Definition~\ref{def:mazur}, which we call \emph{Mazur-type} operators, with properties outlined in Proposition~\ref{pro:mazur} below.

\begin{definition}[Mazur-type operators]\label{def:mazur}
A bounded linear operator $A_1^Y: \ell^1 \to Y$ mapping into a separable infinite-dimensional Banach space $Y$ is called \emph{Mazur-type} if there is a sequence $(\zeta^{(k)})_{k\in\mathbb{N}}$ in $Y$ which is dense in the unit sphere in $Y$, with $\zeta^{(k)}\neq\zeta^{(l)}$ for $k\neq l$, and satisfies, for $x=(x_1,x_2,...) \in \ell^1$, the representation
\begin{equation}\label{eq:mazurstructure}
[A_1 ^Y]\,x=\sum_{k=1}^\infty x_k\,\zeta^{(k)}.
\end{equation}
\end{definition}

\smallskip

Note that, for any $x \in \ell^1$, the weighted sum of $\zeta^{(k)}$-functions in \eqref{eq:mazurstructure} always converges and provides us with an element of $Y$.

\begin{proposition} \label{pro:mazur}
For a separable infinite-dimensional Banach space $Y$, all Mazur-type operators $A_1^Y: \ell^1 \to Y$ in the sense of Definition~\ref{def:mazur} are mapping $\ell^1$ \emph{onto} $Y$.
If $Y$ is reflexive, then all Mazur-type operators $A_1^Y$ are strictly singular, hence of hybrid-type in the sense of Definition~\ref{def:hybrid}, and the associated operator equation \eqref{eq:opeq} with $A:=A_1^Y$ is ill-posed of type~I in the sense of
Definition~\ref{def:new}.
\end{proposition}
\begin{proof}
The existence of bounded linear operators mapping from $\ell^1$ onto any separable infinite-dimensional Banach space $Y$ by using the Mazur-type structure \eqref{eq:mazurstructure} has already been proven by Banach and Mazur, and we refer to  \cite[page~111, item (e)]{Mazur33}. Such proof was presented in a somewhat more accessible form in the context of Theorem~2.3.1 of \cite{AlbKal06}.
The strict singularity of the operator $A_1^Y$ for reflexive Banach spaces $Y$ is a consequence of \cite[Theorem,~item~(b),~p.~334]{GoldThorp63}, because $\ell^1$ does not contain any infinite-dimensional reflexive subspace.
\end{proof}

\begin{remark} \label{rem:l1identity}
Proposition~\ref{pro:mazur} ensures that all Mazur-type operators with $Y=\ell^2$ deliver surjective and strictly singular operators $B=A_1^{\ell^2}: \ell^1 \to \ell^2$ along the lines of
the Goldberg-Thorp example. Since, for all separable Banach spaces $Y$, there are infinitely many sequences $(\zeta^{(k)})_{k\in\mathbb{N}} \subset Y$ with the required properties for Mazur-type operators $A_1^Y$ by formula \eqref{eq:mazurstructure} in Definition~\ref{def:mazur}, infinitely many different versions of the operator $A_1^Y$ exist, and of course also for the operator $B$ from the Goldberg-Thorp example with $Y:=\ell^2$.
\end{remark}

\pagebreak

\begin{theorem} \label{thm:Bprop}
Any surjective bounded linear operator $B: \ell^1 \to \ell^2$ is of hybrid-type, and hence the operator equation \eqref{eq:Beq} is ill-posed of type~I in the sense of Definition~\ref{def:new}.
The operator $B$ has a null-space $\mathcal{N}(B)$, which is uncomplemented in $\ell^1$. Both the dimension and the codimension of the null-space $\mathcal{N}(B)$ are not finite. Consequently,  $B$ is non-injective. The adjoint operator $B^*: \ell^2 \to \ell^\infty$ is injective and hence not strictly singular with an infinite-dimensional closed range $\mathcal{R}(B^*)=\mathcal{N}(B)^\perp$ (annihilator of $\mathcal{N}(B)$). Therefore, the adjoint operator equation
\begin{equation}\label{eq:Beq*}
B^*\,\eta  =z  \quad (\eta \in \ell^2,\; z \in \ell^\infty)
\end{equation}
is well-posed in the sense of Definition~\ref{def:new}.

If $B$ is of Mazur-type attaining the form
\begin{equation}\label{eq:Bstructure}
B\,x:=\sum_{k=1}^\infty x_k\,\zeta^{(k)} \in \ell^2 \qquad \mbox{for} \quad x=(x_1,x_2,...) \in \ell^1,
\end{equation}
the adjoint operator $B^*$ can be explicitly written as
\begin{equation}\label{eq:B*}
B^*\,\eta= (\langle\eta,\zeta^{(1)}\rangle_{\ell^2},\langle\eta,\zeta^{(2)}\rangle_{\ell^2},\ldots)   \in \ell^\infty.
\end{equation}
\end{theorem}
\begin{proof} Since the infinite-dimensional space $\ell^2$ is reflexive, the strict singularity of the operator $B$ from $\ell^1$ onto $\ell^2$ is again a consequence of \cite[Theorem,~item~(b),~p.~334]{GoldThorp63}. Then $B$ is of hybrid-type and the equation \eqref{eq:Beq} ill-posed of type~I. The operator $B$ possesses (due to Proposition~\ref{pro:hybridprop}) a null-space $\mathcal{N}(B)$, which is uncomplemented in $\ell^1$. As a consequence of Corollary~\ref{cor:hybrid}, both the dimension and the codimension of the null-space $\mathcal{N}(B)$ are not finite, and thus $B$ is non-injective.
The injectivity of $B^*$ is an immediate consequence of \cite[Theorem~3.1.22(a)]{Meg98}, which applied to our surjective operator $B$ claims that $B^*$ is an isomorphism from $\ell^2$ onto
a subspace of $\ell^\infty$. The closedness of the range $\mathcal{R}(B^*)$ of $B^*$ and the formula $\mathcal{R}(B^*)=\mathcal{N}(B)^\perp$  are consequences of the closed range theorem (see, e.g., \cite[Theorem of Chapt.VII.5]{Yosida80}). This closedness of the range of $B^*$ together with the injectivity of $B^*$ ensure that the equation \eqref{eq:Beq*} is well-posed in the sense of Definition~\ref{def:new}.
For Mazur-type operators $B$ possessing the structure \eqref{eq:Bstructure}, we have as an immediate consequence of the definition of the adjoint operator that
$$\langle \eta,Bx \rangle_{\ell^2}=\sum_{k=1}^\infty x_k\,\langle \eta,\zeta^{(k)}\rangle_{\ell^2} \quad \mbox{for all} \quad \eta \in \ell^2 \;\;\mbox{and} \;\; x \in \ell^1. $$
This yields immediately the representation \eqref{eq:B*} of $B^*$.
Note that for Mazur-type operators $B$ of the form \eqref{eq:Bstructure} the injectivity of $B^*$ can be proven straightforward when we assume that $B^*\eta_1=B^*\eta_2 \in \ell^\infty$ and $\eta:=\eta_1-\eta_2$ for $\eta,\eta_1,\eta_2 \in \ell^2$. Then we have $(\langle\eta,\zeta^{(1)}\rangle_{\ell^2},\langle\eta,\zeta^{(2)}\rangle_{\ell^2},\ldots)=0$. Since the sequence $(\zeta^{(k)})_{k\in\mathbb{N}}$ is a dense subset of the unit sphere in $\ell^2$,
for an arbitrary element $\zeta \in \ell^2$ with$\|\zeta\|_{\ell^2}=1$, we find that $\langle\eta,\zeta\rangle_{\ell^2}=0$ since $\langle\eta,\zeta^{(k)}\rangle_{\ell^2}=0$ for all $k \in \mathbb{N}$. This, however, yields $\eta=0 \in \ell^2$ and thus the injectivity of $B^*$.
\end{proof}

\begin{remark} \label{rem:nonsurjectivehybrid}
As we have seen by Theorem~\ref{thm:Bprop}, the \emph{surjective} operator $B: \ell^1 \to \ell^2$ from the Goldberg-Thorp example is \emph{of hybrid-type} in the sense of Definition~\ref{def:hybrid}, which means that the operator is \emph{strictly singular} and its range contains an \emph{infinite dimensional closed subspace}. In this remark we note that also \emph{non-surjective} operators may be of hybrid-type. For example, we present the operator $A:\ell^1\times\ell^1\to\ell^2\times\ell^2$, where we set $A:=(B,C)$ with $B$ from the Goldberg-Thorp example and
$C:=\mathcal{E}_1^2$ is the injective embedding operator from $\ell^1$ into $\ell^2$ with range $\mathcal{R}(C)$ dense in $\ell^2$. Then $B$ and $C$ and consequently $A$ are strictly singular and the null-space $\mathcal{N}(A)=\mathcal{N}(B) \times \{0\}$ is uncomplemented. Moreover, the non-closed range $\mathcal{R}(A)=\ell^2 \times \mathcal{R}(C)$ is dense in $\ell^2 \times \ell^2$ and contains the closed infinite-dimensional subspace
$\ell^2\times\{0\}$ of $\ell^2 \times \ell^2$. Consequently, the operator $A$ of this example is of hybrid-type.
\end{remark}

\section{Limitations and chances of regularization approaches for Goldberg-Thorp} \label{sec:regu}

Even if the sequence $(\zeta^{(k)})_{k\in\mathbb{N}}$ forming a countable dense subset of the unit sphere in $\ell^2$ would be known,
an effective regularization strategy for the stable approximate solution to equation \eqref{eq:Beq} is difficult.
In \cite[Theorem~3.3]{FleHof25} it was shown that the associated Tikhonov-type regularization based on the optimization problem
\begin{equation} \label{eq:Bregmin}
 \|Bx-y^\delta\|^2_{\ell^2}+\alpha \|x\|_{\ell^1} \to \min, \quad \mbox{subject to} \; x \in \ell^1,
\end{equation}
as a variant of \eqref{eq:regmin} fails, because regularized solutions $x_\alpha^\delta \in \ell^1$ as minimizers to \eqref{eq:Bregmin} do only exists for noisy data $y^\delta$
from a dense subset of $\ell^2$. An analogous result holds for the existence of minimum-norm solutions $x_{mn}$.

\smallskip

\begin{proposition}
If $y$ is not a multiple of $\zeta^{(k)}$ for some $k$, then there is no minimum-norm solution $x_{mn}$ satisfying
\begin{equation}\label{eq:xmn}
\|x_{mn}\|_{\ell^1} = \min_{\{x\in X:Bx=y\}}\|x\|_{\ell^1}.
\end{equation}
\end{proposition}

\begin{proof}
The minimization problem \eqref{eq:xmn} is equivalent to
\begin{equation*}
\mathcal{N}(B)^\perp\cap\partial\|\cdot\|_{\ell^1}(x_{mn})\neq\emptyset
\end{equation*}
with
\begin{equation*}
\xi\in\partial\|\cdot\|_{\ell^1}(x)\quad\Leftrightarrow\quad\xi_k\begin{cases}
=-1,&\text{if }x_k<0,\\
\in[-1,1],&\text{if }x_k=0,\\
=1,&\text{if }x_k>0
\end{cases}
\;\text{for }k\in\mathbb{N}.
\end{equation*}
From Theorem~\ref{thm:Bprop} we know $\mathcal{N}(B)^\perp=\mathcal{R}(B^\ast)$. Thus, if there exists a minimizer, then there is some $\eta\in\ell^2$ such that $B^\ast\eta\in\partial\|\cdot\|_{\ell^1}(x_{mn})$.
\par
Assume that $x_{mn}$ has at least two non-zero components $[x_{mn}]_m\neq 0$ and $[x_{mn}]_n\neq 0$. Denote the signs of both components by $s_m\in\{-1,1\}$ and $s_n\in\{-1,1\}$, respectively. Then $[B^\ast\eta]_m=s_m$ and $[B^\ast\eta]_n=s_n$ or, equivalently, $\langle\eta,\zeta^{(m)}_{\ell^2}\rangle=s_m$ and $\langle\eta,\zeta^{(n)}\rangle_{\ell^2}=s_n$. Now take a subsequence $(\zeta^{(k_l)})_{l\in\mathbb{N}}$ of $(\zeta^{(k)})_{k\in\mathbb{N}}$ converging to
\begin{equation}
\tilde{\eta}:=\frac{s_m\,\zeta^{(m)}+s_n\,\zeta^{(n)}}{\|s_m\,\zeta^{(m)}+s_n\,\zeta^{(n)}\|_{\ell^2}}.
\end{equation}
Then $\langle\eta,\zeta^{(k_l)}\rangle_{\ell^2}\to\langle\eta,\tilde{\eta}\rangle_{\ell^2}$. From
\begin{equation}
\langle\eta,\tilde{\eta}\rangle_{\ell^2}=\frac{2}{\|s_m\,\zeta^{(m)}+s_n\,\zeta^{(n)}\|_{\ell^2}}
\end{equation}
we see $\langle\eta,\tilde{\eta}\rangle_{\ell^2}\geq 1$ and that $\langle\eta,\tilde{\eta}\rangle_{\ell^2}>1$ holds if and only if $\zeta^{(m)}$ and $\zeta^{(n)}$ are linearly dependent. Thus, $\langle\eta,\tilde{\eta}\rangle_{\ell^2}=1$ is only possible for $\zeta^{(n)}=-\zeta^{(m)}$.
\par
In case $\langle\eta,\tilde{\eta}\rangle_{\ell^2}>1$ we find (large enough) $k_l$ with $\langle\eta,\zeta^{(k_l)}\rangle_{\ell^2}>1$ or, equivalently $[B^\ast\eta]_{k_l}~>~1$. Thus, $B^\ast\,\eta\notin\partial\|\cdot\|_{\ell^1}(x_{mn})$, which shows that $x_{mn}$ can have at most one non-zero component $[x_{mn}]_k$ if $\zeta^{(l)}\neq-\zeta^{(k)}$ for all $l\neq k$. With only one non-zero component in $x_{mn}$ we obtain $Bx_{mn}=[x_{mn}]_k\zeta^{(k)}$.
\par
In case $\langle\eta,\tilde{\eta}\rangle_{\ell^2}=1$, we do not obtain a contradiction (at the moment). That is, $x_{mn}$ may have two non-zero components $[x_{mn}]_m$ and $[x_{mn}]_n$ as long as $\zeta^{(n)}=-\zeta^{(m)}$. But a third non-zero component $[x_{mn}]_l$ is not possible, because corresponding $\zeta^{(l)}$ would have to be equal to both $-\zeta^{(m)}$ and $-\zeta^{(n)}=\zeta^{(m)}$. Thus, $Bx_{mn}=([x_{mn}]_m-[x_{mn}]_n)\zeta^{(m)}$.
\end{proof}

Since the continuous operator $S: \ell^2 \to \ell^1$ with $BSy=y$ for all $y \in \ell^2$ from the Bartle-Graves Theorem
is not explicitly available, this $S$ is not useable  as a (nonlinear) pseudo-inverse for finding stable solutions to equation \eqref{eq:Beq}.
On the other hand, finite-dimensional approximations
$$B_n x:=\sum_{k=1}^n x_k\,\zeta^{(k)}$$
of $B$ and associated least-squares solutions based on the optimization problem
\begin{equation} \label{eq:Bfinite}
 \|B_n x-y^\delta\|^2_{\ell^2} \to \min, \quad \mbox{subject to} \; x \in \ell^1,
\end{equation}
are also not helpful due to the non-injectivity of $B$ and since the null-space $\mathcal{N}(B)$ is not available in an explicit manner. Namely, the finite set $(\zeta^{(1)},\zeta^{(2)},..., \zeta^{(n)})$ does in general not only contain linearly independent functions, and the dimension of the range $\mathcal{R}(B_n)$ may be small (e.g.~only 2) even if $n$ is very large.

The existence and stability deficits of minimizers $x_\alpha^\delta \in \ell^1$ to the optimization problem \eqref{eq:Bregmin} are essentially based on the fact that the hybrid-type operator $B$ fails to be weak$^*$-to-weak continuous as a typical property of type~I ill-posedness (cf.~\cite[Theorem~3.2]{FleHof25}), and consequently \cite[Proposition~3.1]{FleHof25}) does not apply. From a theoretic point of view there is a chance to overcome this drawback when the operator $B$ can be restricted in a reasonable manner to some non-surjective operator with dense range being ill-posed of type~II. We outline this in the following.

\begin{lemma}\label{lem:riesz_basis}
Consider a sequence $(\zeta^{(k)})_{k\in\mathbb{N}}$ that forms a countable dense subset of the unit sphere in $\ell^2$ with $\zeta^{(k)}\neq\zeta^{(l)}$ for $k\neq l$. Moreover let $(e^{(k)})_{k\in\mathbb{N}}$ denote the sequence of usual unit elements belonging to the Hilbert space $\ell^2$. Then there exist a sequence $(k_l)_{l\in\mathbb{N}}$ in $\mathbb{N}$ and an isomorphism $T:\ell^2\to\ell^2$ such that $\zeta^{(k_l)}=Te^{(l)}$ for all $l\in\mathbb{N}$.
\end{lemma}

\begin{proof}
Since the sequence $(\zeta^{(k)})_{k\in\mathbb{N}}$ is dense in the unit sphere of $\ell^2$, we find some $k_l \in \mathbb{N}$ for each $l\in \mathbb{N}$ with
\begin{equation} \label{eq:close}
\|e^{(l)}-\zeta^{(k_l)}\|_{\ell^2}\leq\frac{1}{2l}.
\end{equation}
The corresponding sequence $(\zeta^{(k_l)})_{l\in \mathbb{N}}$ is pairwise disjoint, because $\|e^{(l_1)}-e^{(l_2)}\|_{\ell^2}=\sqrt{2}$ whenever $l_1 \not=l_2$. For each fixed $n\in\mathbb{N}$ and each $(c_1,\ldots,c_n)\in\mathbb{R}^n$ we have
\begin{equation}
\left\|\sum_{l=1}^n c_l\bigl(e^{(l)}-\zeta^{(k_l)}\bigr)\right\|_{\ell^2}
\leq\sum_{l=1}^n |c_l|\bigl\|e^{(l)}-\zeta^{(k_l)}\bigr\|_{\ell^2}
\leq\sqrt{\sum_{l=1}^n|c_l|^2}\sqrt{\sum_{l=1}^n\frac{1}{4l^2}}
\leq\lambda\sqrt{\sum_{l=1}^n|c_l|^2}
\end{equation}
with $\lambda<1$. The assertion of the lemma now follows from \cite[Thm.~13, page~41]{You80}, and we refer to the corresponding definition at page 31 ibid.
\end{proof}

\pagebreak

The subsequence $(\zeta^{(k_l)})_{l\in\mathbb{N}}$ in the lemma is an almost orthonormal basis in the sense of a \emph{Riesz basis}. It is even a \emph{Bari basis}. See \cite{You80} for definitions and context.

\begin{proposition} \label{pro:BU}
There exists a closed infinite-dimensional subspace $U\subseteq\ell^1$ isomorphic to $\ell^1$ such that the restriction $B|_U:\ell^1 \to \ell^2$ of $B$ to $U$ is injective and weak*-to-weak continuous and has dense range $\mathcal{R}(B|_U)$ in $\ell^2$, where the weak*-convergence in $\ell^1$ is understood with regard to the predual space $c_0$.
\end{proposition}

\begin{proof}
Take $(k_l)_{l\in\mathbb{N}}$ and $T$ from Lemma~\ref{lem:riesz_basis} and set
\begin{equation}\label{eq:span}
U:=\overline{\mathrm{span}\,\{e^{(k_l)}:l\in\mathbb{N}\}}^{\ell^1}.
\end{equation}
For $u=\sum_{l=1}^\infty u_l e^{(k_l)}\in U$ with $Bu=0$ we obtain
\begin{equation}
T\left(\sum_{l=1}^\infty u_l e^{(l)}\right)
=\sum_{l=1}^\infty u_l Te^{(l)}
=\sum_{l=1}^\infty u_l \zeta^{(k_l)}
=Bu=0
\end{equation}
and, thus, $\sum_{l=1}^\infty u_l e^{(l)}=0$, which is only possible if $u=0$, because $(e^{(l)})_{l\in\mathbb{N}}$ is an orthonormal basis in $\ell^2$. This proves injectivity of $B|_U$.
\par
To obtain weak*-to-weak continuity, due to \cite[Lemma~9.3]{Flemmingbuch18} it suffices to show that $Be^{(k_l)}$ converges weakly to zero for $l\to\infty$. We have $Be^{(k_l)}=\zeta^{(k_l)}=Te^{(l)}$ and that $(e^{(l)})_{l\in\mathbb{N}}$ converges weakly to zero in $\ell^2$ as $l \to \infty$. Because bounded linear operators are always  weak-to-weak continuous, this proves the assertion.
\par
$U$ obviously is isomorphic to $\ell^1$ and $T^{-1}B|_U$ is the embedding operator $\mathcal{E}_1^2$ from $\ell^1$ into $\ell^2$, the range of which is dense in $\ell^2$. Since the range of $B|_U$ is the image of this dense subspace with respect to the isomorphism $T$, this range is also dense.
\end{proof}

\begin{remark}
Both the lemma and the proposition remain true if $\ell^2$ is replaced by some separable Hilbert space. For $(e^{(l)})_{l\in\mathbb{N}}$ one may choose an arbitrary orthonormal basis.
\end{remark}

Since the operator $B$ restricted to a suitably chosen closed infinite-dimensional subspace $U$ is weak*-weak continuous between $\ell^1$ and $\ell^2$ as a consequence of Proposition~\ref{pro:BU},
the following corollary takes place.

\begin{corollary} \label{cor:Tik}
Let $y \in \mathcal{R}(B|_U)$ and $\|y^\delta-y\|_Y \le \delta$. Then minimizers $x_\alpha^\delta \in U \subseteq \ell^1$ of the optimization problem
\begin{equation} \label{eq:BU}
 \|B|_U\, x-y^\delta\|^2_{\ell^2}+\alpha \|x\|_{\ell^1} \to \min, \quad \mbox{subject to} \; x \in U \subseteq \ell^1,
\end{equation}
(i.e.~Tikhonov-regularized solutions) exist for arbitrary noisy data $y^\delta\in\ell^2$. Moreover, Theorem~9.4 from \cite{Flemmingbuch18} applies and yields convergent approximations when the noise level $\delta>0$ tends to zero.
\end{corollary}

\smallskip

The subspace $U$ has a simple and explicit structure described by (14)
and (16) as long as the sequence $(\zeta_k)_{k\in\mathbb{N}}$ is
explicitly accessible. For now, Mazur-type operators are not of
relevance in practice. Thus, it is not clear whether such
accessibility can be the case. But the results show how to apply
Tikhonov regularization to operators with uncomplemented null-space, in
principle.

\pagebreak

\section*{Appendix}

\textbf{Remarks on the case that} $\overline{\mathcal{R}(A)}^Y$ \textbf{is a proper subspace of} $Y$:
For the case that the closure of the range $\mathcal{R}(A)$ of the bounded linear operator $A$ is a proper subspace of $Y$, i.e.~$\overline{\mathcal{R}(A)}^Y \not=Y$, ill-posedness of type~I and type~II in the sense of Definition~\ref{def:new} may arise in the context of operator equations \eqref{eq:opeq}, but also well-posedness may occur.
We refer to \cite[Sect.~5]{HofKin25} for specific well-posedness and ill-posedness classification details in that case. The special feature here is that the handling of noisy data of the form $y^\delta \in Y \setminus \overline{\mathcal{R}(A)}^Y$ presents additional challenges.
If we consider a direct sum $Y=\overline{\mathcal{R}(A)}^Y \oplus V$, then a continuous linear projection $P:Y \to V$ exists if and only if $V$ is a closed subspace of $Y$, which means that
$\overline{\mathcal{R}(A)}^Y$ is complemented in $Y$. In the case that the closure of the range of $A$ is complemented in $Y$ ($V$ closed), a pseudo-inverse operator $A^\dagger_U: \mathcal{R}(A) \oplus V \to U$, where
$X=\mathcal{N}(A) \oplus U$, is defined with $A^\dagger_U v=0$ for all $v \in V$. But for $y^\delta \notin \mathcal{R}(A) \oplus V$ an element $A^\dagger_U y^\delta$ is not defined. We know this phenomenon from the Moore-Penrose pseudo-inverse in the Hilbert space setting.

\bigskip

\parindent0em{ \textbf{Remarks on quotient maps and its consequences:}}

Let $Q:X\to X/\mathcal{N}(A)$ be the quotient map w.\,r.\,t.\ $\mathcal{N}(A)$ and let $\tilde{A}:X/\mathcal{N}(A)\to Y$ be the injective bounded linear operator such that $A=\tilde{A}Q$. Further let $U$ be a subspace of $X$ such that $X=\mathcal{N}(A)\oplus U$. Denote by $A_U^\dagger:\mathcal{R}(A)\to U$ the corresponding pseudo-inverse of $A$, by $\tilde{A}^{-1}:\mathcal{R}(A)\to X$ the inverse of $\tilde{A}$, and by $Q_U^\dagger:X/\mathcal{N}(A)\to U$ the pseudo-inverse of $Q$ w.\,r.\,t.\ $U$.

\begin{proposition}\label{pro:appendix}
For the operators introduced above we have:
\begin{itemize}
\item[(i)]
$A_U^\dagger$ is bounded if and only both $\tilde{A}^{-1}$ and $Q_U^\dagger$ are bounded.
\item[(ii)]
$\tilde{A}^{-1}$ is bounded if and only if $\mathcal{R}(A)$ is closed.
\item[(iii)]
$\;Q_U^\dagger$ is bounded if and only if $U$ is closed (i.e.~if $\mathcal{N}(A)$ is complemented).
\end{itemize}
\end{proposition}

\begin{proof}
For the non-trivial implication in (ii) observe that $\tilde{A}^{-1}=QA_U^\dagger$ and $Q_U^\dagger=A_U^\dagger\tilde{A}$.
\par
That closedness of $\mathcal{R}(A)$ implies boundedness of $\tilde{A}^{-1}$ in (ii) is a direct consequence of the open mapping theorem. On the other hand, if $\tilde{A}^{-1}$ is bounded, then $\tilde{A}$ is an isomorphism. Thus, $\mathcal{R}(A)$ is a Banach space, because $X/\mathcal{N}(A)$ is a Banach space.
\par
Boundedness of $Q_U^\dagger$ in (iii) is again a consequence of the open mapping theorem. On the other hand, if $Q_U^\dagger$ is bounded, then the restriction $Q|_U$ of $Q$ to $U$ is an isomorphism between $U$ and $X/\mathcal{N}(A)$. Because $X/\mathcal{N}(A)$ is a Banach, then $U$ is also a Banach space.
\end{proof}


\begin{thebibliography}{99}

\bibitem{AlbKal06}
F.~Albiac and N.~J.~Kalton.
\newblock Topics in Banach Space Theory.
\newblock {\em Graduate
  Texts in Mathematics}, vol. 233, Springer, New York, 2006.


\bibitem{Mazur33}
S.~Banach and S.~Mazur.
\newblock  Zur Theorie der linearen Dimension.
\newblock {\em Studia Mathematica}, 4(1):100--112, 1933.

\bibitem{DoRo14}
A.~L.~Dontchev and R.~T.~Rockafellar.
\newblock {\em Implicit Functions and Solution Mappings} (2nd Ed.).
\newblock Springer Series in Operations Research and
  Financial Engineering, Springer, New York, 2014.
\newblock A view from variational analysis.

\bibitem{Flemmingbuch18}
J.~Flemming.
\newblock {\em Variational Source Conditions, Quadratic Inverse Problems,
  Sparsity Promoting Regularization}.
\newblock Frontiers in Mathematics. Birkh\"{a}user/Springer, Cham, 2018.
\newblock New results in modern theory of inverse problems and an application
  in laser optics.


\bibitem {FleHof25}
J.~Flemming and B.~Hofmann.
\newblock New aspects of ill-posedness classification in Banach spaces.
\newblock Paper submitted to {\em Communications in Optimization Theory}, 2026.
\newblock Preprint: arXiv:2511.06690 (Nov.~2025).

\bibitem{GoldThorp63}
S.~Goldberg and E~Thorp.
\newblock On some open questions concerning strictly singular operators.
\newblock {\em Proc.~Amer.~Math.~Soc.}, 14:334--336, 1963.

\bibitem {HofKin25}
B.~Hofmann and S.~Kindermann.
\newblock Classification of ill-posedness for bounded linear operators in {B}anach spaces.
\newblock American Mathematical Society (to appear 2026).
\newblock Preprint: arXiv:2505.12931 (May 2025).

\bibitem{Ivanov63}
V.~K. Ivanov.
\newblock On ill-posed problems (Russian).
\newblock {\em Mat. Sbornik (New Series)}, 61(103):211--223, 1963.

\bibitem{Kato58}
T.~Kato.
\newblock Perturbation theory for nullity, deficiency and other quantities of linear operators.
\newblock {\em J. Analyse Math.}, 6:261--322, 1958.


\bibitem{Meg98}
R.E.~Megginson.
\newblock An Introduction to Banach Space Theory.
\newblock {\em Graduate
  Texts in Mathematics}, vol. 183, Springer, New York, 1998.


\bibitem{Nashed86}
M.~Z.~Nashed.
\newblock A new approach to classification and regularization of ill-posed
  operator equations.
\newblock In {\em {I}nverse and {I}ll-posed {P}roblems ({S}ankt {W}olfgang,
  1986), volume~4 of Notes Rep.~Math.~Sci.~Engrg.}, pages 53--75. Academic
  Press, Boston, MA, 1987.

\bibitem{Schusterbuch12}
T.~Schuster, B.~Kaltenbacher, B.~Hofmann, and K.~S. Kazimierski.
\newblock {\em Regularization methods in {B}anach spaces}, volume~10 of {\em
  Radon Series on Computational and Applied Mathematics}.
\newblock Walter de Gruyter, Berlin, Boston, 2012.

\bibitem{Tao09b}
T.~Tao.
\newblock  245B, Notes 9: The Baire category theorem and its Banach space consequences.
\newblock \url{https://terrytao.wordpress.com/2009/02/01/245b-notes-9-the-baire-category-theorem-and-its-banach-space-consequences/}, 2009.
\newblock "Online accessed 4-Feb-2025".



\bibitem{TikLeoYag98}
A.~N.~Tikhonov, A.~S.~Leonov, and A.~G.~Yagola.
\newblock {\em  Nonlinear Ill-Posed Problems, Volume~1}.
\newblock Chapman \& Hall, London, New York, 1998.


\bibitem{YaKo13}
A.~G.~Yagola and Y.~M.~Korolev.
\newblock Error estimation in ill-posed problems in special cases.
\newblock In {\em Applied Inverse Problems}, Springer Proc.~Math.~Stat., vol.~48, pages 155--164.
\newblock Springer, New York, 2013.


\bibitem{Yosida80}
K.~Yosida.
\newblock {\em Functional Analysis} (6th ed.).
\newblock Springer-Verlag, Berlin-New York, 1980.


\bibitem{You80}
R.~M.~Young.
\newblock An Introduction to Nonharmonic Fourier Series (Chapter~1: Bases in Banach Spaces).
\newblock In {\em Pure and Applied Mathematics}, vol.~93, pages 1--51.
\newblock Academic Press, New York, 1980.



\end{thebibliography}
\end{document}